\documentclass[graybox]{svmult}

\usepackage{mathptmx}       
\usepackage{helvet}         
\usepackage{courier}        
\usepackage{type1cm}        
\usepackage{makeidx}         
\usepackage{graphicx}        
\usepackage{multicol}        
\usepackage[bottom]{footmisc}

\usepackage{amsmath}
\usepackage{amsfonts}

\newcommand{\eps}{\epsilon}

\newcommand{\Var}{{\rm Var}}

\newcommand{\T}{\mathbb{T}}

\begin{document}

\title*{Quantitative limit theorems for local functionals of arithmetic random waves}
 \titlerunning{Quantitative limit theorems for arithmetic random waves}
\author{Giovanni Peccati and Maurizia Rossi}
\institute{Giovanni Peccati \at \email{giovanni.peccati@gmail.com} \and Maurizia Rossi \at \email{maurizia.rossi@uni.lu} \and Unit\'e de Recherche en Math\'ematiques, Universit\'e du Luxembourg \at 6, rue Coudenhove-Kalergi, L-1359 Luxembourg}
%
%
\maketitle

\abstract{We consider Gaussian Laplace eigenfunctions on the two-dimensional flat torus (arithmetic random waves), and provide explicit Berry-Esseen bounds in the 1-Wasserstein distance for the normal and non-normal high-energy approximation of the associated Leray measures and total nodal lengths, respectively. Our results provide substantial extensions (as well as alternative proofs) of findings by Oravecz, Rudnick and Wigman (2007), Krishnapur, Kurlberg and Wigman (2013), and Marinucci, Peccati, Rossi and Wigman (2016). Our techniques involve Wiener-It\^o chaos expansions, integration by parts, as well as some novel estimates on residual terms arising in the chaotic decomposition of geometric quantities that can implicitly be expressed in terms of the coarea formula.} 

\section{Introduction}
\label{sec:intro}

The high-energy analysis of local geometric quantities associated with the \emph{nodal set} of random Laplace eigenfunctions on compact manifolds has gained enormous momentum in recent years, in particular for its connections with challenging open problems in differential geometry (such as {\it Yau's conjecture} \cite{Yau}), and with the striking {cancellation phenomena} detected by Berry in \cite{Berry 2002} --- see the survey \cite{Wsur} for an overview of this domain of research up to the year 2012, and the Introduction of \cite{MPRW} for a review of recent literature. The aim of this paper is to prove {quantitative limit theorems}, in the high-energy limit, for {\it nodal lengths} and {\it Leray measures} (analogous to occupation densities at zero) of Gaussian Laplace eigenfunctions on the two-dimensional flat  torus. These random fields, first introduced by Rudnick and Wigman in \cite{RW2008}, are called {\it arithmetic random waves} and are the main object discussed in the paper. The term `arithmetic' emphasises the fact that, in the two dimensional case, the definition of toral eigenfunctions is inextricable from the problem of enumerating lattice points lying on circles with integer square radius. 

Our results will allow us, in particular, to recover by an alternative (and mostly self-contained) approach the variance estimates from \cite{KKW}, as well as the non-central limit theorems proved in \cite{MPRW}. The core of our approach relies on the use of the Malliavin calculus techniques described in the monograph \cite{NPbook}, as well as on some novel combinatorial estimates for residual terms arising in variance estimates obtained by chaotic expansions.

Although the analysis developed in the present paper focusses on a specific geometric model, we reckon that our techniques might be suitably modified in order to deal with more general geometric objects, whose definitions involve some variation of the area/coarea formulae; for instance, we believe that one could follow a route similar to the one traced below in order to deduce quantitative versions of the non-central limit theorems for phase singularities proved in \cite{DNPR}, as well as to recover the estimates on the nodal variance of toral eigenfunctions and random spherical harmonics, respectively deduced in \cite{RW2008} and \cite{wig}.

\smallskip

From now on, every random object is supposed to be defined on a common probability space $(\Omega, \mathcal F, \mathbb P)$, with $\mathbb E$ denoting expectation with respect to $\mathbb P$.

\vspace{-0.5cm}

\subsection{Setup}
As anticipated, in this paper we are interested in proving quantitative limit theorems for geometric quantities associated with Gaussian eigenfunctions of the Laplace operator $\Delta := \partial^2/\partial x^2_1 +\partial^2/\partial x_2^2$ on the flat torus $\T := \mathbb{R}^2/\mathbb{Z}^2$. In order to introduce our setup, we start by defining
$$
 S := \left \lbrace n\in \mathbb{Z} : n = a^2 + b^2, \, \mbox{ for some } \, a,b\in \mathbb Z \right \rbrace 
$$
to be the set of all numbers that can be written as a sum of two integer squares. It is a standard fact that the eigenvalues of $-\Delta$ are of the form $ 4\pi^2n =:  E_n$, where $n\in S$. The dimension $\mathcal N_n$ of the eigenspace $\mathcal E_n$ corresponding to the eigenvalue $E_n$ coincides with the number $r_2(n)$ of ways  in which $n$ can be expressed as the sum of two integer squares  (taking into account the order of summation). The quantity $\mathcal N_n = r_2(n)$ is a classical object in arithmetics, and is subject to large and erratic fluctuations: for instance, it grows \emph{on average} as $\sqrt{\log n}$ but could be as small as $8$ for an infinite sequence of prime numbers $p_n \equiv 1\, (4)$, or as large as a power of $\log n$ -- see \cite[Section 16.9 and Section 16.10]{HW} for a classical discussion, as well as \cite{KW} for recent advances. We also set
$$
\Lambda_n := \left \lbrace \lambda=(\lambda_1,\lambda_2) \in \mathbb Z^2 : |\lambda|^2 := \lambda_1^2 + \lambda_2^2 = n\right\rbrace
$$
to be the class of all lattice points on the circle of radius $\sqrt n$ (its cardinality $|\Lambda_n|$ equals $\mathcal N_n$). Note that $\Lambda_n$ is invariant w.r.t. rotations around the origin by $k \cdot \pi/2$, where $k$ is any integer. An orthonormal basis $\lbrace \text{e}_\lambda \rbrace_{\lambda\in \Lambda_n}$ for the eigenspace $\mathcal E_n$ is given by the complex exponentials 
$$
\text{e}_\lambda (x) := \exp\left( {i 2\pi\langle \lambda, x\rangle}\right),\quad x=(x_1, x_2)\in \mathbb T.
$$
We now consider a collection (indexed by the set of frequencies $\lambda\in \Lambda_n$) of identically distributed standard complex Gaussian random variables $\lbrace a_\lambda\rbrace_{\lambda\in \Lambda_n}$, that we assume to be independent except for the relations $\overline{a_{\lambda}} = a_{-\lambda}$. We recall that, by definition, every $a_\lambda$ has the form $a_\lambda = b_{\lambda} + i c_{\lambda}$, where $b_{\lambda}, c_{\lambda}$ are i.i.d. real Gaussian random variables with mean zero and variance $1/2$.  We define the \emph{arithmetic random wave} \cite{KKW, MPRW, ORW} of order $n\in S$ to be the real-valued centered Gaussian function
\begin{equation}\label{eq:field}
T_n(x) := \frac{1}{\sqrt{\mathcal N_n}} \sum_{\lambda\in \Lambda_n} a_\lambda e_\lambda(x),\quad x\in \mathbb T;
\end{equation}
from (\ref{eq:field}) it is easily checked that the covariance of $T_n$ is given by, for $x,y\in \mathbb T$,  
\begin{equation}\label{cov_ker}
r_n(x,y) := \mathbb E[T_n(x)\cdot T_n(y)]= \frac{1}{\mathcal N_n} \sum_{\lambda\in \Lambda_n} \cos(2\pi \langle \lambda, x-y\rangle)=:r_n(x-y). 
\end{equation}
Note that $r_n(0)=1$, i.e. $T_n(x)$ has unit variance for every $x\in \mathbb T$. Moreover, as emphasised in the right-hand side (r.h.s.) of \eqref{cov_ker}, the field $T_n$ is \emph{stationary}, in the sense that its covariance (\ref{cov_ker}) depends only on the difference $x-y$.  From now on, without loss of generality, we assume that $T_n$ is stochastically independent of $T_m$ for $n\neq m$.

\smallskip

For $n\in S$, we will focus on the \emph{zero set} 
$
T_n^{-1}(0) = \left \lbrace x\in \mathbb T : T_n (x) =0\right \rbrace;
$
recall that, according e.g. to \cite{C}, with probability one $T_n^{-1}(0)$ consists of the union of a finite number of rectifiable (random) curves, called \emph{nodal lines}, containing a finite set of isolated singular points. In this manuscript, we are more specifically interested in the following two \emph{local} functionals associated with the nodal set $T_n^{-1}(0)$:
\begin{enumerate}
\item the \emph{Leray (or microcanonical) measure}  defined as \cite[(1.1)]{ORW}
\begin{equation}\label{leray def}
\mathcal Z_n := \lim_{\varepsilon\to 0} \frac{1}{2\varepsilon}\, \text{meas} \left \lbrace x\in \mathbb T : |T_n(x)| < \varepsilon \right \rbrace,
\end{equation}
 where `${\rm meas}$' stands for the Lebesgue measure on $\mathbb{T}$, and the limit is in the sense of convergence in probability;
\item the (total) \emph{nodal length} $\mathcal L_n$, given by (see \cite{KKW})
\begin{equation}\label{length def}
\mathcal L_n := \text{length}\left (T_n^{-1}(0) \right );
\end{equation}
for technical reasons, we will sometimes need to consider {\it restricted nodal lengths}, that are defined as follows: for every measurable $Q\subset \mathbb{T}$,
\begin{equation}\label{e:lnl}
\mathcal L_n(Q) := \text{length}\left (T_n^{-1}(0)\cap Q \right ).
\end{equation}
\end{enumerate}
We observe that, in the jargon of stochastic calculus, the quantity $\mathcal Z_n$ corresponds to the \emph{occupation density at zero} of $T_n$ -- see \cite{GH} for a classical reference on the subject. 

As already discussed, our aim is to establish quantitative limit theorems for both $\mathcal{Z}_n$ and $\mathcal{L}_n$ \emph{in the high-energy limit}, that is, when $\mathcal{N}_n\to + \infty$.

\smallskip

\noindent{\bf Notation.} Given two positive sequences $\{a_n\}_{n\in S}$, $\{b_n\}_{n\in S}$ we will write:
\vspace{-0.2cm}

\begin{enumerate}
\item $a_n \ll b_n$, if there exists a finite constant $C>0$ such that $a_n \le C b_n$, $\forall n\in S$. Similarly, $a_n {\ll}_{\alpha}\, b_n$ (resp. $a_n {\ll}_{\alpha, \beta}\, b_n$) will mean that $C$ depends on $\alpha$ (resp. $\alpha, \beta$);
\item ``$a_n {\ll} b_n$, as $\mathcal{N}_n \to +\infty$''  (or equivalently ``$a_n = O(b_n)$, as $\mathcal{N}_n \to +\infty$'' ) if, for every subequence $\{n\}\subset S$ such that $\mathcal{N}_{n}\to \infty$, the ratio $a_{n} / b_{n}$ is asymptotically bounded. Similarly, ``$a_n {\ll}_{\alpha}\, b_n$, as $\mathcal{N}_n \to +\infty$'', (resp. ``$a_n {\ll}_{\alpha, \beta}\, b_n$, as $\mathcal{N}_n \to +\infty$'') will mean that the bounding constant depends on $\alpha$ (resp. $\alpha, \beta$);
\item $a_n \asymp b_n$ (resp. $a_n \asymp b_n$, $\mathcal{N}_n \to +\infty$) if both $a_n {\ll} b_n$ and $b_n {\ll} a_n$ (resp. $a_n {\ll} b_n $ and $b_n {\ll} a_n$, as $\mathcal{N}_n \to +\infty$) hold;
\item$a_n = o(b_n)$ if $a_{n} / b_{n}\to 0$ as $n \to +\infty$ (and analogously for subsequences); 
\item$a_n \sim b_n$ if $a_{n} / b_{n}\to 1$ as $n\to +\infty$ (and analogously for subsequences). 
\end{enumerate}

\vspace{-0.7cm}

\subsection{Previous work}\label{sec:previous}

\vspace{-0.3cm}

\subsubsection{Leray measure}

The Leray measure in (\ref{leray def}) was investigated by Oravecz, Rudnick and Wigman \cite{ORW}. They found that \cite[Theorem 4.1]{ORW}, for every $n\in S$,
\begin{equation}\label{mediaLeray}
\mathbb E \left [\mathcal Z_n \right ] = \frac{1}{\sqrt{2\pi}},
\end{equation}
i.e. the expected Leray measure is constant, 
and moreover \cite[Theorem 1.1]{ORW},
\begin{equation}\label{varLeray}
\text{Var}(\mathcal Z_n) = \frac{1}{4\pi \mathcal N_n} + O\left( \frac{1}{\mathcal N_n^2} \right).
\end{equation}
In particular, the asymptotic behaviour of the variance, as $\mathcal{N}_n\to +\infty$, is independent of the distribution of lattice points lying on the circle of radius $\sqrt n$. 

\subsubsection{Nodal length} 

The expected nodal length was computed in \cite{RW2008} to be, for $n\in S$, 
\begin{equation}\label{exp_length}
\mathbb E[\mathcal L_n] = \frac{1}{2\sqrt 2}\sqrt{E_n}.
\end{equation}
Computing the nodal variance is a subtler issue, and its asymptotic behaviour (in the high-energy limit) was fully characterized in \cite{KKW} as follows. We start by observing that the set $\Lambda_n$ induces a probability measure $\mu_n$ on the unit circle $\mathbb S^1$, given by $\mu_n := \frac{1}{\mathcal N_n} \sum_{\lambda\in \Lambda_n} \delta_{\lambda/\sqrt n},$ where $\delta_{\theta}$ denotes the Dirac mass at $\theta\in \mathbb S^1$.  One crucial fact is that, although there exists a density-$1$ subsequence $\lbrace n_j\rbrace \subset S$ such that 
$
\mu_{n_j} \Rightarrow d\theta/2\pi
$, as $j\to +\infty$\footnote{From now on, $\Rightarrow$ denotes weak-$\ast$ convergence of measures and $d\theta$ the uniform measure on $\mathbb S^1$}, there is an infinity of other weak-$\ast$ adherent points for the sequence $\lbrace \mu_n\rbrace_{n\in S}$ --- see \cite{KW} for a partial classification. 
In particular, for every $\eta\in [-1,1]$, there exists  a subsequence $\lbrace n_j\rbrace \subset S$ (see \cite{KKW, KW}) such that 
\begin{equation}\label{conv_eta}
\widehat{\mu_{n_j}}(4) \to \eta,\quad \text{as } j\to +\infty,
\end{equation}
where, for a probability measure $\mu$ on the unit circle, the symbol $\widehat \mu(4)$ stands for the fourth Fourier coefficient
$
\widehat \mu(4) := \int_{\mathbb S^1} \theta^{-4}\,d\mu(\theta)
$. Krishnapur, Kurlberg and Wigman in \cite{KKW} found that, as $\mathcal N_n\to +\infty$, 
\begin{equation}\label{igor_var}
\text{Var}(\mathcal L_n) = c_n \frac{E_n}{\mathcal N_n^2}(1 + o(1)),
\end{equation}
where $\displaystyle{ c_n := (1+\widehat{\mu_n}(4)^2)/{512}. }$
Such a result is in stark contrast with (\ref{varLeray}): indeed, it shows that the asymptotic variance of the nodal length  multiplicatively depends on the distribution of lattice points lying on the circle of radius $\sqrt n$,  via the fluctuations of the squared Fourier coefficient $\widehat{\mu_n}(4)^2$;
this also entails that the order of magnitude of the variance is $E_n/\mathcal N_n^2$, since the sequence $\lbrace |\widehat{\mu_n}(4)| \rbrace_n$ is bounded by $1$.  Plainly, in order to obtain an  asymptotic behaviour in (\ref{igor_var}) that has no multiplicative corrections, one needs to  extract a subsequence $\lbrace n_j\rbrace \subset S$ such that $\mathcal N_{n_j}\to +\infty$ and $|\widehat{\mu_{n_j}}(4)|$ converges  to some $\eta\in [0,1]$; in this case, one deduces that $\text{Var}(\mathcal L_{n_j}) \sim c(\eta) {E_{n_j}} / {\mathcal N_{n_j}^2},$ where $ c(\eta) :=(1 +\eta^2)/512$. Note that if $\mu_{n_j} \Rightarrow \mu$, then $\widehat{\mu_{n_j}}(4)\to \widehat \mu(4)$. By (\ref{conv_eta}), the possible values of the constant $c(\eta)$ span therefore the whole interval $[1/512, 1/128]$.

The second order behavior of the nodal length was investigated in \cite{MPRW}. Let us define, for $\eta\in [0,1]$, the random variable
\begin{equation}\label{Meta}
\mathcal M_\eta:= \frac{1}{2\sqrt{1+\eta^2}}\left (2 - (1+\eta)X_1^2 - (1-\eta) X_2^2 \right ),
\end{equation}
where $X_1$, $X_2$ are i.i.d. standard Gaussians. Note that $\mathcal M_\eta$ is invariant in law under the transformation $\eta\mapsto -\eta$, so that if $\eta\in [-1,0)$ we define $\mathcal M_{\eta} := \mathcal M_{-\eta}$. 

Theorem 1.1 in \cite{MPRW} states that for $\lbrace n_j\rbrace\subset S$ such that $\mathcal N_{n_j}\to +\infty$ and $|\widehat{\mu_{n_j}}(4)|\to \eta$, as $j\to +\infty$, one has that
\begin{equation}\label{conv_law}
 \widetilde{\mathcal L}_{n_j}\mathop{\to}^d \mathcal M_\eta,
\end{equation}
where $\displaystyle{\mathop{\to}^d}$ denotes convergence in distribution and, for $n\in S$,
\begin{equation}\label{normaliz nodal}
 \widetilde{\mathcal L}_{n}:= \frac{\mathcal L_{n} - \mathbb E\left [\mathcal L_{n} \right ]}{\sqrt{\text{Var}\left(\mathcal L_{n} \right )}}
\end{equation}
is the normalized nodal length. 
 Note that (\ref{conv_law}) is a non-universal and non central limit theorem: indeed, for $\eta \ne \eta'$ the (non Gaussian) laws of the random variables $\mathcal M_\eta$ and $\mathcal M_{\eta'}$ in (\ref{Meta}) have different supports. 

\subsection{Main results}\label{sec:main}

The main purpose of this paper is to prove quantitative limit theorems for local functionals of nodal sets of arithmetic random waves, such as the Leray measure in (\ref{leray def}) and the nodal length in (\ref{length def}).
We will work with the 1-Wasserstein distance (see e.g. \cite[\S C]{NPbook} and the references therein). Given two random variables $X, Y$ whose laws are $\mu_X$ and $\mu_Y$, respectively, the Wasserstein distance between $\mu_X$ and $\mu_Y$,
written $d_{\text{W}}(X,Y)$, is defined as
\begin{equation*}\label{d def gen2}
d_{\text{W}}\left(X, Y\right):= \inf_{(A,B)}\mathbb E\left[ \left | A-B \right |  \right ],
\end{equation*}
 where the infimum runs over all pairs of random variables $(A,B)$ with marginal laws $\mu_X $ and $\mu_Y$, respectively. We will mainly use the dual representation 
\begin{equation}\label{d def gen}
d_{\text{W}}\left(X, Y\right) = \sup_{h\in \mathcal H} \left |\mathbb E\left[h(X) - h(Y)\right]\right |,
\end{equation}
where $\mathcal H$ denotes the class of Lipschitz functions $h:\mathbb R\to \mathbb R$ whose Lipschitz constant is less or equal than $1$. Relation \eqref{d def gen} implies in particular that, if $d_{\rm W}(X_n, X)\to 0$, then 
$\displaystyle{X_n \mathop{\to}^d  X}$ (the converse implication is false in general). Our first result is a \emph{uniform} bound for the Wasserstein distance between the normalized Leray measure
\begin{equation}\label{normalizedL}
\widetilde{\mathcal Z_n} := \frac{\mathcal Z_n - \mathbb E\left [ \mathcal Z_n \right ]}{\sqrt{\text{Var}(\mathcal Z_n)}}
\end{equation}
and a standard Gaussian random variable.
\begin{theorem}\label{mainthL}
We have that, on $S$,
\begin{equation}\label{rateWleray}
d_{\text{W}}\left(\widetilde {\mathcal Z_n},  Z\right) \ll \mathcal N_n^{-1/2},
\end{equation}
where $\widetilde{\mathcal Z_n}$ is defined in (\ref{normalizedL}), and $Z\sim \mathcal N(0,1)$ is a standard Gaussian random variable. In particular, if $\lbrace n_j\rbrace\subset S$ is such that $\mathcal N_{n_j}\to +\infty$, then $\displaystyle{ \widetilde {\mathcal Z_{n_j}} \mathop{\to}^d Z.}$
\end{theorem}
The following theorem deals with nodal lengths, providing a quantitative counterpart to the convergence result stated in (\ref{conv_law}). 
\begin{theorem}\label{mainth1}
 As $\mathcal N_{n}\to +\infty$, one has that
\begin{equation}\label{maineq}
d_{\rm W}\left(\widetilde {\mathcal L}_{n}, \mathcal M_\eta \right) \ll  \mathcal N_{n}^{-1/4}\,   \vee \,  \left | \left |\widehat{\mu_{n}}(4) \right | - \eta \right |^{1/2},
\end{equation}
where $\widetilde {\mathcal L}_{n}$ and $\mathcal M_\eta$ are defined, respectively, in (\ref{conv_law})  and (\ref{Meta}). 
\end{theorem}
 Note that (\ref{maineq}) entails the limit theorem (\ref{conv_law}): it is important to observe that, while the arguments exploited in \cite{MPRW} directly used the variance estimates in \cite{KKW}, the proof of (\ref{conv_law}) provided in the present paper is basically self-contained, except for the use of a highly non-trivial combinatorial estimate by Bombieri and Bourgain \cite{BB}, appearing in our proof of Lemma 2 below --- see Section 5. We also notice that the bound (\ref{rateWleray}) for the Leray measure is uniform on $S$, whereas the bound (\ref{maineq}) for the nodal length holds asymptotically, and depends on the angular distribution of lattice points lying on the circle of radius $\sqrt{n}$. 

 By combining the arguments used in the proofs of Theorem \ref{mainthL} and Theorem \ref{mainth1} with the content of \cite[Section 4.2]{MPRW}, one can also deduce the following multidimensional limit theorem, yielding in particular a form of {\it asymptotic dependence} between Leray measures and nodal lenghts. 
\begin{corollary}
Let $\lbrace n_j\rbrace\subset S$ be such that $\mathcal N_{n_j}\to +\infty$ and $|\widehat{\mu_{n_j} }(4)| \to \eta\in [0,1]$, then
\begin{equation*}
\left ( \widetilde {\mathcal Z}_{n_j}, \, \, \widetilde{\mathcal L}_{n_j} \right) \mathop{\to}^d \left ( Z_1, \frac{q(Z)}{\sqrt{1+\eta^2}} \right ),
\end{equation*}
where $Z=Z(\eta)=(Z_1, Z_2, Z_3, Z_4)$ is a centered Gaussian vector with covariance matrix 
$$
\left (\begin{matrix}
\vspace{0.1cm}
&1 &\frac12 & \frac 12  &0\\
\vspace{0.1cm}
&\frac12 &\frac{3+\eta}{8} & \frac{1-\eta}{8}&0\\
\vspace{0.1cm}
&\frac12 &\frac{1-\eta}{8} & \frac{3+\eta}{8}&0\\
\vspace{0.1cm}
&0 &0 & 0 &\frac{1-\eta}{8} 
\end{matrix} \right ),
$$
and $q$ is the polynomial $q(z_1,z_2,z_3, z_4):= 1+ z_1^2 -2z_2^2 -2z_3^2 - 4z_4^2. $
\end{corollary}

The details of the proof are left to the reader.

\section{Outline of our approach}

\subsection{About the proofs of the main results}\label{subsec:ontheproof}

In order to prove Theorem \ref{mainthL} and Theorem \ref{mainth1}, we pervasively use \emph{chaotic expansion} techniques (see \S \ref{sec:chaos}). 
Since both $\mathcal Z_n$ in (\ref{leray def}) and $\mathcal L_n$ in (\ref{length def}) are finite-variance functionals of a Gaussian field, 
they can be written as a series, converging in $L^2(\mathbb P)$, whose terms can be explicitly found:
\begin{equation}\label{chaos_exp}
\mathcal Z_n = \sum_{q=0}^{+\infty} \mathcal Z_n[2q],\qquad \mathcal L_n = \sum_{q=0}^{+\infty} \mathcal L_n[2q].
\end{equation}
For each $q\ge 0$, the random variable $\mathcal Z_n[2q]$ (resp. $\mathcal L_n[2q]$) is the orthogonal projection of $\mathcal Z_n$ (resp. $\mathcal L_n$) onto the so-called \emph{Wiener chaos} of order $2q$, that will be denoted by $C_{2q}$. Since $C_0 = \mathbb R$, we have $\mathcal Z_n[0] = \mathbb E[\mathcal Z_n]$ and $\mathcal L_n[0] = \mathbb E[\mathcal L_n]$; moreover, chaoses of different orders are orthogonal in $L^2(\mathbb P)$. 

\subsubsection{On the proof of Theorem \ref{mainthL}}

We first need the following result, that will be proved in \S \ref{secProofThL}. 
\begin{proposition}\label{var2leray}
For $n\in S$ (cf. (\ref{varLeray}))
\begin{equation}\label{varBella}
\text{Var}(\mathcal Z_n[2]) = \frac{1}{4\pi \mathcal N_n}.
\end{equation}
Moreover, for every $K\geq 2$,
\begin{equation}\label{stimaVarsup0}
\sum_{q\ge K}\text{Var}(\mathcal Z_n[2q]) \ll_K \int_{\mathbb{T}} r_n(x)^{2K} dx \quad \mbox{on $S$};
\end{equation}
in particular, for $K=2$,
\begin{equation}\label{stimaVarsup}
\sum_{q\ge 2}\text{Var}(\mathcal Z_n[2q]) \ll \mathcal N_n^{-2}. 
\end{equation}
\end{proposition}
Proposition \ref{var2leray} gives an {alternative} proof of (\ref{varLeray}) via chaotic expansions and entails also that, as $\mathcal N_n\to +\infty$,
$$
\frac{{\mathcal Z_n}- \mathbb E[\mathcal Z_n]}{\sqrt{\text{Var}(\mathcal Z_n)}}= \frac{\mathcal Z_n[2]}{\sqrt{\text{Var}(\mathcal Z_n[2])}} + o_{\mathbb P}(1),
$$
where $o_{\mathbb P}(1)$ denotes a sequence converging to $0$ in probability. In particular, the Leray measure and its second chaotic component have the same {asymptotic} behavior, since different order Wiener chaoses are orthogonal. 
 Let us now introduce some more notation. 
If $\sqrt{n}$ is an integer, we define
\begin{equation*}
\Lambda_n^+ := \lbrace \lambda=(\lambda_1, \lambda_2)\in \Lambda_n : \lambda_2 >0 \rbrace \cup \lbrace (\sqrt{n}, 0)\rbrace,
\end{equation*}
otherwise 
$
\Lambda_n^+ := \lbrace \lambda=(\lambda_1, \lambda_2)\in \Lambda_n : \lambda_2 >0 \rbrace
$. 
Note that $|\Lambda_n^+ | = \mathcal N_n /2$ in both cases. 
\begin{lemma}\label{lem:2chaos}
For $n\in S$ 
\begin{equation*}
\mathcal Z_n[2] = -\frac{1}{\sqrt{2\pi}} \frac{1}{\mathcal N_n} \sum_{\lambda\in \Lambda_n^+} (|a_\lambda|^2-1). 
\end{equation*}
\end{lemma}
Lemma \ref{lem:2chaos}, proven in \S \ref{secProofThL} below, states that the second chaotic component is (proportional to) a sum of independent random variables. To conclude the proof of Theorem \ref{mainthL}, note that we can write
\begin{equation}\label{disugW}
d_{\text{W}}\left ( \widetilde{\mathcal Z_n}, Z \right )\le d_{\text{W}}\left (\widetilde{\mathcal Z_n},   \widetilde{\mathcal Z_n}[2] \right ) + d_{\text{W}}\left (\widetilde{\mathcal Z_n}[2], Z\right ),
\end{equation}
where $\widetilde{\mathcal Z_n}[2]:= \mathcal Z_n[2]/ \sqrt{\text{Var}(\mathcal Z_n[2])}$. The first term on the right-hand side of (\ref{disugW}) may be bounded by (\ref{stimaVarsup}), whereas for the second term 
standard results apply, thanks to Lemma \ref{lem:2chaos}. 

\subsubsection{On the proof of Theorem \ref{mainth1}}

The proof of Theorem \ref{mainth1} is similar to that one of Theorem \ref{mainthL}. 
In \cite{MPRW} it has been shown that $\mathcal L_n[2]=0$ for every $n\in S$, and moreover that, as $\mathcal N_n\to +\infty$,
\begin{equation}\label{lim var}
\text{Var}(\mathcal L_n)\sim \text{Var}(\mathcal L_n[4]),
\end{equation}
by proving that the asymptotic variance of $\mathcal L_n[4]$ equals the r.h.s. of (\ref{igor_var}). 
The result stated in (\ref{lim var}) and the orthogonality properties of Wiener chaoses entail that the fourth chaotic component and the total length have the same asymptotic behavior i.e., as $\mathcal N_n\to +\infty$,
\begin{equation}\label{= law}
\frac{\mathcal L_n - \mathbb E[\mathcal L_n]}{\sqrt{\text{Var}(\mathcal L_n)}} = \frac{\mathcal L_n[4]}{\sqrt{\text{Var}(\mathcal L_n[4])}} + o_{\mathbb P}(1),
\end{equation}
where $o_{\mathbb P}(1)$ denotes a sequence converging to $0$ in probability. Finally, in \cite{MPRW} it was shown that $\mathcal L_n[4]$ can be written as a \emph{polynomial transform} of an asymptotically Gaussian random vector, so that the same convergence as in (\ref{conv_law}) holds when replacing the total nodal length with its fourth chaotic component. 

Now let $h:\mathbb R\to \mathbb R$ be a $1$-Lipschitz function and $\lbrace n_j\rbrace_j\subset S$ be such that $\mathcal N_{n_j}\to +\infty$ and $|\widehat{\mu_{n_j}}(4)|\to \eta$, as $j\to +\infty$. Bearing in mind (\ref{d def gen}) and (\ref{= law}), we write, by virtue of the triangle inequality,
\begin{equation}\label{eqapp}
\left | \mathbb E\left[ h  (\widetilde{\mathcal L}_{n_j} ) - h (\mathcal M_\eta )  \right ] \right | \le \mathbb E\left[\left | h  (\widetilde{\mathcal L}_{n_j} ) - h (\widetilde{\mathcal L}_{n_j}[4]) \right | \right ]  +\left | \mathbb E\left[ h  (\widetilde{\mathcal L}_{n_j}[4] ) - h (\mathcal M_\eta )  \right] \right |,
\end{equation}
where $\widetilde{\mathcal L}_{n_j}[4]:= \mathcal L_{n_j}[4]/\sqrt{\text{Var}(\mathcal L_{n_j}[4])}$. 
Let us deal with the first term on the r.h.s. of (\ref{eqapp}).
\begin{proposition}\label{prop1}
Let $h:\mathbb R\to \mathbb R$ be a $1$-Lipschitz function and $\lbrace n\rbrace\subset S$ such that $\mathcal N_n\to +\infty$, then 
\begin{equation}\label{bound1}
\mathbb E\left[\left | h  (\widetilde{\mathcal L}_{n} ) - h (\widetilde{\mathcal L}_{n}[4]) \right | \right ] \ll \mathcal N_n^{-1/4}. 
\end{equation}
\end{proposition}
In order to prove Proposition \ref{prop1} in \S \ref{sec:proofProp1}, we need to control the behavior of the variance tail $\sum_{q\ge 3}\text{Var}(\mathcal L_{n}[2q])$. 
\begin{lemma}\label{lem1}
For every $K\geq 3$, on $S$ we have 
\begin{equation}\label{e:vartailk}
\sum_{q\ge K}\text{Var}(\mathcal L_{n}[2q]) \ll_K E_n \int_\mathbb{T} r_n(x)^{2K} \, dx;
\end{equation}
 in particular, if $\mathcal N_n\to +\infty$, 
\begin{equation}\label{var tail}
\sum_{q\ge 3}\text{Var}(\mathcal L_{n}[2q]) \ll E_n\, \mathcal N_{n}^{-5/2}.  
\end{equation}
\end{lemma}
The proof of Lemma \ref{lem1}  is considerably more delicate than that of (\ref{stimaVarsup0}), see \S \ref{sec:proofProp1}, and together with a precise investigation of the fourth chaotic component gives also
an {alternative} proof of (\ref{igor_var}) via chaotic expansions. 

For the second term on the r.h.s. of (\ref{eqapp}), recall from above that in \cite{MPRW} it was shown that $\mathcal L_n[4]$ can be written as a polynomial transform $p$ of a random vector, say $W(n)$, 
which is asymptotically Gaussian. Let us denote by $Z$ this limiting vector. Then, we can reformulate our problem as the estimation of the distributional distance between $p(W(n_j))$ and $p(Z)$, the latter distributed as $\mathcal M_\eta$ in (\ref{Meta}).  
To prove the following in \S \ref{sec:proofProp2} we can take advantage of some results in \cite{CD, CDart}. 
\begin{proposition}\label{prop2}
Let $h:\mathbb R\to \mathbb R$ be a $1$-Lipschitz function and let $|\widehat{\mu_{n_j}}(4)|\to \eta\in [0,1]$, as $\mathcal N_{n_j}\to +\infty$, then 
\begin{equation}\label{ineq2}
 \left |\mathbb E\left[ h  (\widetilde{\mathcal L}_{n_j}[4] ) - h (\mathcal M_\eta )  \right] \right | \ll \mathcal N_{n_j}^{-1/4} \, \vee \, ||\widehat{\mu_{n_j}}(4)| - \eta|^{1/2}. 
\end{equation}
\end{proposition}
Proposition \ref{prop1} and Proposition \ref{prop2} allow one to prove Theorem \ref{mainth1} in \S\ref{sec:proofProp2}, bearing in mind (\ref{d def gen}) and (\ref{eqapp}). We now state and prove a technical result, which is a key tool for the proofs of Theorems \ref{mainthL} and \ref{mainth1}. 

\subsection{A technical result}\label{ss:technical}

Some of the main bounds in our paper will follow from technical estimates involving pairs of cubes contained in the Cartesian product $\mathbb{T}\times \mathbb{T}$, that will be implicitly classified (for every fixed $n\in S$) according to the behaviour of the mapping $(x,y) \mapsto \mathbb{E}[T_n(x)\cdot T_n(y)] = r_n(x-y)$ appearing in (\ref{cov_ker}). 

\smallskip

\noindent{\bf Notation.} For every integer $M\geq 1$, we denote by $\mathcal{Q}(M)$ the partition of $\mathbb{T}$ obtained by translating in the directions $k/M$ ($k\in \mathbb{Z}^2$) the square $Q_0 = Q_0(M) := [0, 1/M) \times [0, 1/M)$. Note that, by construction, $| \mathcal{Q}(M)| = M^2$.

\smallskip

Now we fix, for the rest of the paper, a small number $\epsilon \in (0, 10^{-3})$. The following statement unifies several estimates taken from \cite[\S 6.1]{DNPR} (yielding Point 4), and \cite[\S 4.1]{KKW} (yielding Point 5) and \cite[\S 6.1]{ORW}. A sketch of the proof is provided for the sake of completeness.

\begin{proposition}\label{p:cubes} There exists a mapping $ M : S\to \mathbb{N} : n \mapsto M(n)$, as well as sets $G_0(n), G_1(n) \subset \mathcal{Q}(M(n)) \times \mathcal{Q}(M(n))$ with the following properties:
\begin{enumerate}

\item  there exist constants $1< c_1<c_2 <\infty$ such that $ c_1 E_n \leq M(n)^2 \leq c_2 E_n$ for every $n \in S$;

\item for every $n\in S$, $G_0(n) \cap G_1(n) = \emptyset$ and $G_0(n) \cup G_1(n) =\mathcal{Q}(M(n)) \times \mathcal{Q}(M(n))$; 

\item $(Q,Q') \in G_0(n)$ if and only if for every $(x,y)\in Q\times Q'$, and for every choice of $i\in \{1,2\}$ and $(i,j)\in \{1,2\}^2$,
\begin{equation}\label{e:es1}
| r_n(x-y) |, \, |\partial_i r_n(x-y)/\sqrt{n}\,  |, \, |\partial_{i,j} r_n(x-y)/n |  \leq \epsilon,
\end{equation}
where $\partial_i r_n := \partial/\partial x_i\, r_n$ and $\partial_{i,j} := \partial/\partial_{x_i x_j}\, r_n. $
\item for every fixed $K \geq 2$, one has that
\begin{equation}\label{e:masada0} 
 | G_1(n) | \ll_{\epsilon, K}\, E_n^2 \int_\mathbb{T} | r_n(x)|^{2K}\,dx;
\end{equation}
\item adopting the notation \eqref{e:lnl}, one has that 
\begin{equation}\label{e:masada}
{\rm Var}(\mathcal{L}_n(Q_0) )\, \ll \, 1/E_n;
\end{equation} 
\item for every fixed $q\geq 2$, one has that
\begin{equation}\label{svolta2}  
\int_{\hat{Q}_0} r_n(x) ^{2q} dx \, \ll\,  \frac{1}{2E_n(q+1)}\left( 1 - \left(1-\frac{E_n}{M(n)^2}\right)^{q+1}\right),
\end{equation} 
where $\hat{Q}_0  := Q_0-Q_0$, and the constant involved in the above estimates is independent of $q$. 
\end{enumerate}
\end{proposition}

\begin{proof}[Sketch] \smartqed The combination of Points 1--4 in the above statement corresponds to a slight variation of \cite[Lemma 6.3]{DNPR}. Both estimates \eqref{e:masada} and \eqref{svolta2} follow from the fact that $\hat{Q}_0$ is contained in the union of four adjacent {\it positive singular cubes}, in the sense of \cite[Definition 6.3]{ORW}\footnote{Indeed, each one of the four cubes composing $\hat{Q}_0$ is such that its boundary contains the point $x_0 = (0,0)$, and the singularity in the sense of \cite[Definition 6.3]{ORW} follows by the continuity of trigonometric functions.}. Using such a representation of $\hat{Q}_0$, in order to prove \eqref{e:masada} it is indeed sufficient to apply the same arguments as in \cite[\S 4.1]{KKW} for deducing that, defining the rescaled correlation 2-points function $K_2$ as in \cite[formula (29)]{KKW},
$$
{\rm Var}(\mathcal{L}_n(Q_0) ) = E_n \int_{Q_0} \int_{Q_0} K_2(x-y) dxdy \leq \frac{E_n}{M(n)^2} \int_{\hat{Q}_0} K_2(x) dx \, \ll \, \frac{1}{E_n}.
$$
Finally, arguing as in \cite[\S 6.5]{ORW}, we infer that $r_n(x)^2 \leq 1 - E_n\|x-x_0\|^2$, where $x_0 = (0,0)$ and the estimate holds for every $x\in \hat{Q}_0$, yielding in turn the relations
\begin{eqnarray*}\notag
\int_{\hat{Q}_0} r_n(x)^{2q}\,dx \leq \int_{\|x-x_0\|\ll  \frac{1}{M}} \left( 1 - E_n\|x-x_0\|^2\right)^q\,dx\ll \int_0^{\frac{1}{M}} r(1-E_nr^2)^q\,dr \nonumber \\
= \frac{1}{2E_n} \frac{1}{q+1}\left( 1 - \left(1-\frac{E_n}{M(n)^2}\right)^{q+1}   \right), \notag
\end{eqnarray*}
and therefore the desired conclusion. \qed
\end{proof}

\section{Local functionals and Wiener chaos}\label{sec:chaos} 

As mentioned in \S \ref{subsec:ontheproof}, for the proof of our main results we need the notion of Wiener-It\^o chaotic expansions for non-linear functionals of Gaussian fields. In what follows, we will present it in a simplified form adapted to our situation; we refer the reader to \cite[\S 2.2]{NPbook} for a complete discussion. 

\subsection{Wiener Chaos}
\label{ss:berryintro}

Let $\phi$ denote the standard Gaussian density on $\mathbb R$ and 
$L^2(\mathbb{R}, \mathcal{B}(\mathbb{R}), \phi(t)dt) =: L^2(\phi)$ the space of square integrable functions on the real line w.r.t. the Gaussian measure $ \phi(t)dt$. The sequence of normalized Hermite polynomials $\{(k!)^{-1/2}H_k\}_{k\ge 0}$ is a complete orthonormal basis of $ L^2(\phi)$; recall \cite[Definition 1.4.1]{NPbook} that 
they are defined recursively as follows: $H_0 \equiv 1$, and, for $k\geq 1$, $H_{k}(t) = tH_{k-1}(t) - H'_{k-1}(t),\, t\in \mathbb R.$ Recall now the definition of the arithmetic random waves (\ref{eq:field}), and observe that it involves a family of complex-valued Gaussian random variables $\{a_\lambda : \lambda\in \mathbb{Z}^2\}$ with the following properties: {(i)} $a_\lambda=b_\lambda+ic_\lambda$, where $b_\lambda$ and $c_\lambda$ are two independent real-valued centered Gaussian random variables with variance $1/2$; {(ii)} $a_\lambda$ and $a_{\lambda'}$ are independent whenever $\lambda' \notin\{ \lambda, -\lambda\}$, and {(iii)} $a_\lambda = \overline{a_{-\lambda}}$. Consider now the space of all real finite linear combinations of random variables $\xi$ of the form $\xi = z \, a_\lambda + \overline{z} \, a_{-\lambda},$ where $\lambda\in \mathbb{Z}^2$ and $z\in \mathbb{C}$. Let us denote by 
${\bf A}$ its closure in $L^2(\mathbb{P})$; it turns out that ${\bf A}$ is a real centered Gaussian Hilbert subspace
of $ L^2(\mathbb{P})$.
\begin{definition}\label{d:chaos} Let $q$ be a nonnegative integer; the $q$-th \emph{Wiener chaos} associated with ${\bf A}$, denoted by $C_q$, is the closure in $L^2(\mathbb{P})$ of all real finite linear combinations of random variables of the form
$$
H_{p_1}(\xi_1)\cdot H_{p_2}(\xi_2)\cdots H_{p_k}(\xi_k)
$$
for $k\ge 1$, where the integers $p_1,...,p_k \geq 0$ satisfy $p_1+\cdots+p_k = q$, and $(\xi_1,...,\xi_k)$ is a standard real Gaussian vector extracted
from ${\bf A}$ (note that, in particular, $C_0 = \mathbb{R}$).
\end{definition}
It is well-known (see \cite[\S 2.2]{NPbook}) that $C_q $ and $C_m$ are orthogonal in $L^2(\mathbb{P})$ whenever $q\neq m$, and moreover $L^2(\Omega, \sigma({\bf A}), \mathbb{P}) = \bigoplus_{q\geq 0} C_q;$
equivalently, every real-valued functional $F$ of ${\bf A}$ can be (uniquely) represented in the form
\begin{equation}\label{e:chaos2}
F = \sum_{q=0}^\infty F[q],
\end{equation}
where $F[q]$ is the orthogonal projection of $F$ onto $C_q$, and the series converges in $L^2(\mathbb{P})$. Plainly, $F[0]= \mathbb E [F]$.

\subsection{Chaotic expansion of $\mathcal Z_n$} 

We can rewrite (\ref{leray def}) as
\begin{equation}\label{conv_as}
\mathcal Z_n = \lim_{\varepsilon \to 0} \frac{1}{2\varepsilon} \int_{\mathbb T} 1_{[-\varepsilon, \varepsilon]} (T_n(x))\,dx =: \lim_{\varepsilon \to 0} \mathcal Z_n^\varepsilon,
\end{equation}
and hence formally represent the Leray measure as 
\begin{equation}\label{formalLeray}
\mathcal Z_n = \int_{\mathbb T} \delta_0 (T_n(x))\,dx,
\end{equation}
where $\delta_0$ denotes the Dirac mass at $0\in \mathbb R$. Let us now consider the sequence of coefficients $\lbrace \beta_{2q}\rbrace_{q\ge 0}$ defined as 
\begin{equation}\label{e:beta}
\beta_{2q}:= \frac{1}{\sqrt{2\pi}}H_{2q}(0),
\end{equation}
where $H_{2q}$ denotes the $2q$-th Hermite polynomial, as before. It can be seen as the sequence of coefficients corresponding to the (formal) chaotic expansion of the Dirac mass. 

The following result concerns the chaotic expansion of the Leray measure in (\ref{formalLeray}) and will be proved in the Appendix. 
\begin{lemma}\label{lemChaosLeray}
For $n\in S$, one has that $\mathcal Z_n\in L^2(\mathbb{P})$, and the chaotic expansion of $\mathcal Z_n$ is 
\begin{equation}\label{e:chaoszn}
\mathcal Z_n = \sum_{q=0}^{+\infty} \mathcal Z_n[2q] = \sum_{q=0}^{+\infty} \frac{\beta_{2q}}{(2q)!} \int_{\mathbb T} H_{2q}(T_n(x))\,dx,
\end{equation}
where $\beta_{2q}$ is given in (\ref{e:beta}), and the convergence of the above series holds in $L^2(\mathbb P)$. 
\end{lemma} 

\subsection{Chaotic expansion of $\mathcal L_n$}
\label{expan}

We recall now from \cite{MPRW} the chaotic expansion (\ref{e:chaos2}) for the nodal length. 
First, $\mathcal L_n$ in (\ref{length def}) admits the following integral representation
\begin{equation}\label{integral1}
\mathcal L_n = \int_{\mathbb T} \delta_0(T_n(x))| \nabla T_n(x)|\,dx,
\end{equation}
where $\delta_0$ still denotes the Dirac mass at $0\in \mathbb R$ and $\nabla T_n$ the gradient of $T_n$; more precisely, $\nabla T_n = (\partial_1 T_n, \partial_2 T_n)$ with $\partial_i := \partial/\partial x_i$ for $i=1,2$.  The integral in \eqref{integral1} has to be interpreted in the sense that, for any sequence of bounded probability densities $\{g_k\}$ such that the associated probabilities weakly converge to $\delta_0$, one has that  $\int_{\mathbb T} g_k(T_n(x))| \nabla T_n(x)|\,dx \to \mathcal L_n$ in $L^2(\mathbb{P})$. A straightforward differentiation of the definition \eqref{eq:field} of $T_n$ yields, for $j=1,2$
\begin{equation}\label{e:partial}
\partial_j T_n(x) = \frac{2\pi i}{\sqrt{\mathcal{N}_n} }\sum_{(\lambda_1,\lambda_2)\in \Lambda_n} \lambda_j a_\lambda e_\lambda(x).
\end{equation}
Hence the random fields $T_{n},\partial_{1} T_n,\partial_{2} T_n$ viewed as collections of
Gaussian random variables
indexed by $x\in\mathbb T$ are all lying in ${\bf A}$, i.e. for every $x\in\mathbb T$ we have
\begin{equation*}
T_{n}(x),\, \partial_{1}T_{n}(x), \, \partial_{2}T_{n}(x) \in \bf A.
\end{equation*}
It has been proved in \cite{KKW} that the random variables $T_n(x), \partial_1 T_n(x), \partial_2 T_n(x)$ are independent for fixed $x\in \mathbb T$,
and for $i=1,2$ 
\begin{equation}\label{var der}
\text{Var}(\partial_i T_n(x)) = \frac{E_n}{2}. 
\end{equation}
We can write from (\ref{integral1}), keeping in mind (\ref{var der}),
\begin{equation}\label{integral2}
\mathcal L_n = \sqrt{\frac{E_n}{2}} \int_{\mathbb T} \delta_0(T_n(x))| \widetilde \nabla T_n(x)|\,dx,
\end{equation}
with $\widetilde \nabla T_n := (\widetilde \partial_1 T_n, \widetilde \partial_2 T_n)$ and for $i=1,2$, $\widetilde \partial_i := \partial_i / \sqrt{E_n/2}$. Note that $\widetilde \partial_i T_n(x)$ has unit variance for every $x\in \mathbb T$. 

Equation (\ref{integral1}), or equivalently (\ref{integral2}), explicitly represents the nodal length as a (finite-variance) non-linear functional of a Gaussian field. To recall its chaotic expansion, we need (\ref{e:beta}) and moreover have to introduce the collection of coefficients
$\{\alpha_{2n,2m} : n,m\geq 1\}$, that is related to the Hermite expansion of the norm $| \cdot |$ in $\mathbb R^2$:
\begin{equation}\label{e:alpha}
\alpha_{2n,2m}=\sqrt{\frac{\pi}{2}}\frac{(2n)!(2m)!}{n!
m!}\frac{1}{2^{n+m}} p_{n+m}\left (\frac14 \right),
\end{equation}
where for $N=0, 1, 2, \dots $ and $x\in \mathbb R$
\begin{equation*}
\displaylines{ p_{N}(x) :=\sum_{j=0}^{N}(-1)^{j}\cdot(-1)^{N}{N
\choose j}\ \ \frac{(2j+1)!}{(j!)^2} x^j, }
\end{equation*}
$\frac{(2j+1)!}{(j!)^2}$ being the so-called ``swinging factorial"
restricted to odd indices.
From \cite[Proposition 3.2]{MPRW}, we have for $q=2$ or $q=2m+1$ odd ($m\ge 1$)
$
\mathcal L_n [q] \equiv 0
$, 
and for $q\geq 2$
\begin{eqnarray}\label{e:pp}
\nonumber \mathcal L_n [2q]= \sqrt{\frac{4\pi^2n}{2}}\sum_{u=0}^{q}\sum_{k=0}^{u}
\frac{\alpha _{2k,2u-2k}\beta _{2q-2u}
}{(2k)!(2u-2k)!(2q-2u)!} \times \nonumber \\
\times \int_{\mathbb T}\!\! H_{2q-2u}(T_n(x))
H_{2k}(\widetilde\partial_1  T_n(x))H_{2u-2k}(\widetilde \partial_2
T_n(x))\,dx.
\end{eqnarray}
The Wiener-It\^o chaotic expansion of $\mathcal L_n$ is hence
\begin{eqnarray*}
\mathcal L_n = \mathbb E [{\mathcal L_n}] + \sqrt{\frac{4\pi^2n}{2}}\sum_{q=2}^{+\infty}\sum_{u=0}^{q}\sum_{k=0}^{u}
\frac{\alpha _{2k,2u-2k}\beta _{2q-2u}
}{(2k)!(2u-2k)!(2q-2u)!}\times\\
\nonumber
\times \int_{\mathbb T}H_{2q-2u}(T_n(x))
H_{2k}(\widetilde \partial_1 T_n(x))H_{2u-2k}(\widetilde \partial_2
 T_n(x))\,dx,
\end{eqnarray*}
with convergence in $L^2(\mathbb P)$.
\subsubsection{Fourth chaotic components}\label{subsec:4chaos}
In this part we investigate the fourth chaotic component $\mathcal L_n[4]$ (from (\ref{e:pp}) with $q=2$), recalling also some facts from \cite{MPRW}. 

Consider, for $n\in S$, the four-dimensional random vector $W=W(n)$ given by
\begin{equation*}
W(n) = \left ( \begin{matrix} &W_1(n) \\ &W_2(n) \\ &W_3(n) \\ &W_4(n) \end{matrix}\right) := \frac{1}{\sqrt{\mathcal N_n/2}} \sum_{\lambda\in \Lambda_n^+} (|a_\lambda|^2 -1) \left (\begin{matrix} &1 \\ &\lambda_1^2/n \\ &\lambda_2^2/n \\ &\lambda_1 \lambda_2 /n \end{matrix}\right),
\end{equation*}
whose covariance matrix is 
\begin{equation}\label{e:sign}
\Sigma_n =\left(
\begin{array}{cccc}
1 & \frac{1}{2} & \frac{1}{2} & 0 \\
\frac{1}{2} & \frac{3+\widehat{\mu_n}(4)}{8}  & \frac{1-\widehat{\mu_n}(4)}{8}  & 0 \\
\frac{1}{2} & \frac{1-\widehat{\mu_n}(4)}{8}  & \frac{3+\widehat{\mu_n}(4)}{8}  & 0 \\
0 & 0 & 0 & \frac{1-\widehat{\mu_n}(4)}{8}
\end{array}%
\right),
\end{equation}%
see \cite[Lemma 4.1]{MPRW}. 
Note that for every $n\in S$
\begin{equation}\label{LC}
W_2(n) + W_3(n) = W_1(n). 
\end{equation}
The following will be proved in the Appendix and is a finer version of \cite[Lemma 4.2]{MPRW}.
\begin{lemma}\label{lem:4chaos}
For every $n\in S$, 
\begin{equation}\label{eq1bella}
\mathcal L_n[4] = \sqrt{\frac{E_n}{\mathcal N_n^2}}\frac{1}{\sqrt{512}}\left(W_1^2 - 2W_2^2 -2W_3^2 -4W_4^2 +\frac{1}{2} \frac{1}{\mathcal N_n}\sum_{\lambda\in \Lambda_n} |a_\lambda|^4 \right),
\end{equation}
and moreover, 
\begin{equation}\label{exact_var4}
\text{Var}(\mathcal L_{n}[4]) = \frac{E_n}{512 \mathcal N_n^2}\left( 1 + \widehat{\mu_n}(4)^2 + \frac{34}{\mathcal N_n}  \right). 
\end{equation}
\end{lemma}
It is worth noticing that Lemma \ref{lem1} and (\ref{exact_var4}) immediately give an alternative proof of (\ref{igor_var}) via chaotic expansion. 

We recall here from \cite[Lemma 4.3]{MPRW} that, for 
$\{n_j\}\subseteq S$ such that $\mathcal N_{n_j} \to +\infty$ and $\widehat{\mu}_{n_j}(4) \to \eta\in [-1,1]$, as $j\to \infty$, the following CLT holds:
\begin{eqnarray}\label{e:punz}
W(n_j) \stackrel{ d}{\to} Z = Z(\eta) = \left(
\begin{array}{c}
Z_{1} \\
Z_{2} \\
Z_{3} \\
Z_{4}%
\end{array}%
\right ),
\end{eqnarray}
where $Z(\eta)$ is a centered Gaussian vector with covariance
\begin{equation}\label{e:sig}
\Sigma=\Sigma(\eta) =\left(
\begin{array}{cccc}
1 & \frac{1}{2} & \frac{1}{2} & 0 \\
\frac{1}{2} & \frac{3+\eta}{8}  & \frac{1-\eta}{8}  & 0 \\
\frac{1}{2} & \frac{1-\eta}{8}  & \frac{3+\eta}{8}  & 0 \\
0 & 0 & 0 & \frac{1-\eta}{8}
\end{array}%
\right).
\end{equation}%
The eigenvalues of $\Sigma $ are $ 0,\frac{3}{2},%
\frac{1-\eta}{8},\frac{1+\eta}{4}$ and hence, in particular, $\Sigma$ is singular.
Moreover, 
\begin{equation*}
\label{eq:Lctild->Meta2}
\frac{\mathcal{L}_{n_j}[4] }{\sqrt{\text{Var}(\mathcal{L}_{n_j}[4] ) }} \stackrel{ d}{\longrightarrow} \mathcal{M}_{|\eta|},
\end{equation*}
where $\mathcal M_{|\eta|}$ is defined as in (\ref{Meta}), see \cite[Proposition 2.2]{MPRW}.

\section{Proof of Theorem \ref{mainthL}}\label{secProofThL}

Note first that, from (\ref{e:beta}) and (\ref{e:chaoszn}) for $q=0$
$$
\mathcal Z_n[0]= \beta_0 = \frac{1}{\sqrt{2\pi}},
$$
cf. (\ref{mediaLeray}). Let us now focus on the second chaotic component of the Leray measure in (\ref{e:chaoszn}), by proving Lemma \ref{lem:2chaos}. 
\begin{proof}[Lemma \ref{lem:2chaos}] \smartqed 
By (\ref{e:beta}) and (\ref{e:chaoszn}) for $q=1$, recalling that $H_2(t)=t^2 -1$,
$$
\mathcal Z_n[2] =- \frac{1}{2\sqrt{2\pi}} \int_{\mathbb T} (T_n(x)^2-1)\,dx. 
$$
Finally, (\ref{eq:field}) allows us to conclude the proof. \qed
\end{proof}
We can now prove Proposition \ref{var2leray}. 
\begin{proof}[Proposition \ref{var2leray}]
\smartqed
From Lemma \ref{lem:2chaos}, straightforward computations based on independence yield that 
$$
\text{Var}(\mathcal Z_n[2]) = \frac{1}{4\pi \mathcal N_n},
$$
that is (\ref{varBella}). We can rewrite (\ref{stimaVarsup0}) as 
\begin{equation}\label{svolta}
\sum_{q\ge K} \text{Var}(\mathcal Z_n[2q])=\sum_{q=K}^{+\infty} \frac{\beta^2_{2q}}{(2q)!}\int_{\mathbb T} r_n(x)^{2q}\,dx \, \ll \, \int_{\mathbb T} r_n(x)^{2K}\,dx
\end{equation}
(note that the first equality in \eqref{svolta} is a direct consequence of \eqref{e:chaoszn}, \cite[Proposition 1.4.2]{NPbook}  and stationarity of $T_n$). 
Our proof of the second equality in (\ref{svolta}), which is (\ref{stimaVarsup0}), uses the content of Proposition \ref{p:cubes}. We can rewrite the middle term in \eqref{svolta}, by stationarity of $T_n$, as
\begin{eqnarray}\label{svolta1}
\nonumber \sum_{q=K}^{+\infty} \frac{\beta^2_{2q}}{(2q)!}\int_{\mathbb T} r_n(x)^{2q}\,dx &=&\sum_{q=K}^{+\infty} \frac{\beta^2_{2q}}{(2q)!}\int_{\mathbb T}\int_{\mathbb T} r_n(x-y)^{2q}\,dxdy \\ \nonumber &=& 
\sum_{q=K}^{+\infty} \frac{\beta^2_{2q}}{(2q)!}\sum_{(Q,Q') \in G_0(n)} \int_{Q}\int_{Q'} r_n(x-y)^{2q}\,dxdy \\
&& + \sum_{q=K}^{+\infty} \frac{\beta^2_{2q}}{(2q)!}\sum_{(Q,Q') \in G_1(n)} \int_{Q}\int_{Q'} r_n(x-y)^{2q}\,dxdy \notag \\
&=:& A(n)+B(n).
\end{eqnarray}
Using Point 3 in Proposition \ref{p:cubes} one infers that 
\begin{eqnarray}\label{stimaA}
\nonumber A(n) \le \sum_{q=K}^{+\infty} \frac{\beta^2_{2q}}{(2q)!}\epsilon^{2q-2K}\sum_{(Q,Q') \in G_0(n)} \int_{Q}\int_{Q'} r_n(x-y)^{2K}\,dxdy \\
\le \sum_{q=K}^{+\infty} \frac{\beta^2_{2q}}{(2q)!}\epsilon^{2q-2K}\int_{\mathbb{T}} r_n(x)^{2K}\,dx.
\end{eqnarray}
It is easy to check that, since $\epsilon\in (0,1)$, then
$$
\sum_{q=1}^{+\infty} \frac{\beta^2_{2q}}{(2q)!}\epsilon^{2q}<\infty
$$
(indeed, $\beta^2_{2q}/(2q)! \asymp 1/\sqrt{q}$, as $q\to \infty$), finally yielding 
\begin{equation}\label{contrib1}
A(n) \, \ll_{\epsilon, K} \,  \int_{\mathbb T} r_n(x)^{2K}\,dx.
\end{equation}
Let us now focus on $B(n)$. For every pair $(Q,Q')\in G_1(n)$ and every $q\geq 1$, we can use Cauchy-Schwartz inequality and then exploit the stationarity of $T_n$ to write
\begin{eqnarray*}
 \int_{Q}\int_{Q'} r_n(x-y)^{2q}\,dxdy &=&  (2q)!^{-1} \mathbb{E}\left[ \int_Q H_{2q}(T_n(x)) dx \int_{Q'} H_{2q}(T_n(y)) dy\right] \\ 
& \leq & (2q)!^{-1}  {\rm Var}\left(\int_{Q_0} H_{2q}(T_n(x)) dx \right) =  \int_{Q_0}\int_{Q_0} r_n(x-y)^{2q}\,dxdy\\  &\leq & \int_{Q_0}dy\int_{\hat{Q}_0} r_n(x)^{2q}\,dx \ll \frac{1}{E_n}\int_{\hat{Q}_0} r_n(x)^{2q}\,dx,  
\end{eqnarray*}
where the constant involved in the last estimate is independent of $q$. Using \eqref{e:masada0} and \eqref{svolta2}, one therefore deduces that
\begin{eqnarray}\label{stimaB}
\nonumber &B(n)& \ll   \int_{\mathbb{T}} r_n(x)^{2K}\, dx \times \sum_{q=K}^{+\infty} \frac{\beta^2_{2q}}{(2q)!} \frac{1}{q+1}\left( 1 - \left(1-\frac{E_n}{M^2}\right)^{q+1}   \right)\\
&=&\!\!\! \!\!\!  \int_{\mathbb{T}} r_n(x)^{2K}\, dx \times \left( \sum_{q=K}^{+\infty} \frac{\beta^2_{2q}}{(2q)!}  \frac{1}{q+1} - \sum_{q=K}^{+\infty} \frac{\beta^2_{2q}}{(2q)!}  \frac{1}{q+1}\left(1-\frac{E_n}{M^2}\right)^{q+1} \right).
\end{eqnarray}
Since the series appearing in the above expression are both convergent, substituting (\ref{stimaA}) and (\ref{stimaB}) in (\ref{svolta1}), bearing in mind (\ref{svolta}), we immediately have (\ref{stimaVarsup0}). To prove (\ref{stimaVarsup}), it suffices to recall (from (\ref{cov_ker})) that for every integer $K\ge 1$
\begin{equation}\label{r equal}
\int_{\mathbb T} r_n(x)^{2K}\,dx = \frac{|S_{2K}(n)|}{\mathcal N_n^{2K}},
\end{equation}
where 
\begin{equation}\label{S def}
S_{2K}(n) = \lbrace (\lambda_1, \lambda_2, \dots, \lambda_{2K})\in \Lambda_n^{2K} : \lambda_1 + \lambda_2 + \cdots + \lambda_{2K} =0 \rbrace. 
\end{equation}
For $K=2$, from \cite{KKW} we have 
\begin{equation}\label{S4}
|S_{4}(n)| = 3\mathcal N_n (\mathcal N_n -1),
\end{equation}
so that substituting (\ref{S4}) into (\ref{stimaVarsup0}) for $K=2$, bearing in mind (\ref{r equal}), we obtain (\ref{stimaVarsup}). \qed
\end{proof}
This section ends with the proof of Theorem \ref{mainthL}. 
\begin{proof}[Theorem \ref{mainthL}]
\smartqed 
We write for (\ref{disugW})
\begin{eqnarray}\label{ciao}
d_{\text{W}}\left ( \widetilde{\mathcal Z_n}, Z \right )\le d_{\text{W}}\left (\widetilde{\mathcal Z_n},   \widetilde{\mathcal Z_n}[2] \right ) + d_{\text{W}}\left (\widetilde{\mathcal Z_n}[2], Z\right )\nonumber\\
\le d_{\text{W}}\left (\widetilde{\mathcal Z_n},   \mathcal Z_n[2]/\sqrt{\text{Var}(\mathcal Z_n)} \right ) + d_{\text{W}}\left (\mathcal Z_n[2]/\sqrt{\text{Var}(\mathcal Z_n)}, \widetilde{\mathcal Z_n}[2]  \right )\nonumber \\+ d_{\text{W}}\left (\widetilde{\mathcal Z_n}[2], Z\right ).
\end{eqnarray}
Bearing in mind (\ref{d def gen}), the first term on the r.h.s. of (\ref{ciao}) can be dealt with as follows
\begin{equation}\label{facile1}
d_{\text{W}}\left (\widetilde{\mathcal Z_n},   \mathcal Z_n[2]/\sqrt{\text{Var}(\mathcal Z_n)} \right )\le \sqrt{\frac{\sum_{q\ge 2} \text{Var}(\mathcal Z_n[2q])}{\text{Var}(\mathcal Z_n)}}\ll \mathcal N_n^{-1/2},
\end{equation}
where the last estimate comes from (\ref{stimaVarsup}), and the trivial lower bound for the total variance $\Var(\mathcal Z_n) \ge \Var(\mathcal Z_n[2])$. For the second term on the r.h.s. of (\ref{ciao}) we have 
\begin{equation}\label{facile2}
d_{\text{W}}\left (\mathcal Z_n[2]/\sqrt{\text{Var}(\mathcal Z_n)}, \widetilde{\mathcal Z_n}[2]  \right )
\le \left | \frac{1}{\sqrt{1 + \frac{\sum_{q\ge 2} \text{Var}(\mathcal Z_n[2q])}{\text{Var}(\mathcal Z_n[2])}}} - 1 \right | \ll \mathcal N_n^{-1},
\end{equation}
where we used (\ref{varBella}) and (\ref{stimaVarsup}).  Thanks to Lemma \ref{lem:2chaos}, we can now deal with the last term in (\ref{ciao}) by using the standard Berry-Esseen theorem (see e.g. \cite[Section 3.7]{NPbook}).\qed
\end{proof}

\section{Proof of Proposition \ref{prop1}}\label{sec:proofProp1}

In this section we will prove Proposition \ref{prop1}. Let us first give the proof of Lemma \ref{lem1}.
 \begin{proof}[Lemma \ref{lem1}] 
\smartqed
Fix $K\geq 3$, and recall the notation \eqref{e:lnl}. In order to simplify the discussion, for every $n\in S$ and given $Q\in \mathcal{Q}(M(n))$, we shall denote by $\mathcal{L}_n(Q\, ; \,  \geq 2K)$, the projection of the random variable $\mathcal{L}_n(Q)$ onto the direct sum of chaoses $\bigoplus_{q\geq K} C_{2q}$. For the l.h.s. of (\ref{e:vartailk}) we write 
$$
\sum_{q\geq K}{\rm Var} (\mathcal{L}_n[2q]) = \sum_{(Q,Q')} {\rm Cov}\left\{ \mathcal{L}_n(Q\, ; \,  \geq 2K), \mathcal{L}_n(Q'\, ; \,  \geq 2K)\right\},
$$
where the sum runs over the cartesian product $\mathcal{Q}(M(n))\times \mathcal{Q}(M(n))$. We now write $\sum_{(Q,Q')} = \sum_{(Q,Q')\in G_0(n)} + \sum_{(Q,Q')\in G_1(n)}$, and study separately the two terms. By virtue of Cauchy-Schwarz and stationarity of $T_n$, one has that 
\begin{eqnarray*}
\sum_{(Q,Q')\in G_1(n)} {\rm Cov}\left\{ \mathcal{L}_n(Q\, ; \,  \geq 2K), \mathcal{L}_n(Q'\, ; \,  \geq 2K)\right\} &\leq&  | G_1(n) | {\rm Var}(\mathcal{L}_n(Q_0))\\
& \ll & E_n \int_{\mathbb{T}} r_n(x)^{2K}\, dx,
\end{eqnarray*} 
where we have used \eqref{e:masada0} and \eqref{e:masada}, together with the fact that, by orthogonality, ${\rm Var}(\mathcal{L}_n(Q\, ; \,  \geq 2K))\leq {\rm Var}(\mathcal{L}_n(Q)) = {\rm Var}(\mathcal{L}_n(Q_0))$. The rest of the proof follows closely the arguments rehearsed in \cite[\S 6.2.2]{DNPR}. For all $Q\in \mathcal{Q}(M(n))$, we write
\begin{eqnarray*}
&& \mathcal{L}_n(Q\, ; \,  \geq 2K) =\sqrt{\frac{E_n}{2}} \sum_{q\ge K}\sum_{i_1+i_2+i_3 =2q} \frac{\beta_{i_1}\alpha_{i_2,i_3}}{i_1! i_2! i_3! }\times\\
&& \quad\quad\quad\quad\quad\times \int_{Q} H_{i_1}(T_n(x)) H_{i_2}(\widetilde\partial_1  T_n(x)) H_{i_3}(\widetilde \partial_2 T_n(x)) \,dx,
\end{eqnarray*}
where the sum runs over all even integers $i_1, i_2,i_3\geq 0$. We have
\begin{equation}\label{lungo1}
\begin{split}
& \left| \sum_{(Q,Q')\in G_0(n) }  {\rm Cov} \left (\mathcal{L}_n(Q\, ; \,  \geq 2K), \, \mathcal{L}_n(Q'\, ; \,  \geq 2K) \right) \right|\cr
&\le E_n \sum_{q\ge 2K} \sum_{i_1+i_2+i_3=2q}\sum_{a_1+a_2+a_3+=2q} \left |\frac{\beta_{i_1}\alpha_{i_2,i_3}}{i_1! i_2! i_3! }\right |\cdot \left |\frac{\beta_{a_1}\alpha_{a_2,a_3}}{a_1! a_2! a_3!}\right |\cr
&\quad \quad \quad\times\Big|\sum_{(Q,Q')\in G_0(n)} \int_Q \int_{Q'}  \mathbb{E}\Big [ H_{i_1}(T_n(x)) H_{i_2}(\widetilde{\partial_1}  T_n(x)) H_{i_3}(\widetilde{\partial_2  }T_n(x))\cr
&\quad \quad \quad\quad \quad \quad\times H_{a_1}(T_n(y)) H_{a_2}(\widetilde{\partial_1}  T_n(y)) H_{a_3}(\widetilde{\partial_2} T_n(y))\Big ]\,dxdy\Big|.
\end{split}
\end{equation}
For $n\in S$, we now introduce the notation
\begin{eqnarray*}
&& (X_0(x), X_1(x), X_2(x)) := (T_{n}(x),\, \widetilde\partial_{1} T_{n}(x), \, \widetilde\partial_{2} T_{n}(x)), \quad x\in \T.
\end{eqnarray*}
Applying the Leonov-Shyraev formulae for cumulants, in a form analogous to \cite[Proposition 2.2]{DNPR}, we infer that
\begin{eqnarray}\label{lungo2}
&&\left| \sum_{(Q,Q')\in G_0(n) }  {\rm Cov} \left (\mathcal{L}_n(Q\, ; \,  \geq 2K), \, \mathcal{L}_n(Q'\, ; \,  \geq 2K) \right) \right|\\
&&\le E_n \sum_{q\ge 2K} \sum_{i_1+i_2+i_3=2q}\sum_{a_1+a_2+a_3=2q} \left |\frac{\beta_{i_1}\alpha_{i_2,i_3}}{i_1! i_2! i_3! }\right |\cdot \left |\frac{\beta_{a_1}\alpha_{a_2,a_3}}{a_1! a_2! a_3! }\right | \notag \\
&&\quad\times {\bf 1}_{\{i_1+i_2+i_3 = a_1+a_2+a_3\}}  \Big| U(i_1,i_2,i_3  ; a_1,a_2,a_3 )  \Big| ,\notag \\
&& :=E_n \times Z,\label{e:rol}
\end{eqnarray}
where each summand $U = U(i_1,i_2,i_3 ; a_1,a_2,a_3 )$ is the sum of at most $(2q)!$ terms of the type 
\begin{equation}\label{e:sergio}
u =\sum_{(Q,Q')\in G_0(n)} \int_Q\int_{Q'} \prod_{u=1}^{2q} R_{l_u, k_u}(x,y) \, dxdy,
\end{equation}
with $k_u, l_u \in \{0,1,2\}$ and, for $l,k=0,1,2$ and $x,y\in \mathbb T$, and we set
$$
R_{l,k}(x,y):= \mathbb{E}\left[X_l(x) X_k(y)   \right] = R_{l,k}(x-y),
$$
where the last equality (with obvious notation) emphasises the fact that $R_{l,k}(x,y)$ only depends on the difference $x-y$. We will also exploit the following relation, valid for every {even integer $p$}:
\begin{equation}\label{deriv}
\int_{\mathbb T} R_{l,k}(x)^{p}\,dx \leq \int_{\mathbb T} r(x)^p\,dx;
\end{equation}
also, for $x,y\in \T$, one has $|R_{l,k}(x-y)|\leq 1$, and, for $(x,y)\in Q\times Q'$,
\begin{equation}\label{epsuno}
|R_{l,k}(x-y)|\leq \eps.
\end{equation}
Using the properties of $G_0(n)$ put forward in Proposition \ref{p:cubes}, as well as the fact that the sum defining $Z$ in \eqref{e:rol} involves indices $q\geq 2K$, one infers that, for $u$ as in \eqref{e:sergio}, 
\begin{eqnarray*}
|u| &\leq& \eps_1^{2q-2K} \sum_{(Q,Q')\in G_0(n) } \int_Q\int_{Q'} \prod_{u=1}^{2K} \left |R_{l_u, k_u}(x,y)\right| dxdy\\
&\leq &  \eps_1^{2q-2K} \int_\T \prod_{u=1}^{2K} \left |R_{l_u, k_u}(x)\right| dx \leq  \eps_1^{2q-2K} R_n(2K),
\end{eqnarray*}
where $R_n(2K) = \int_\T r_n(x)^{2K}\, dx$, and we have applied a generalised H\"older inequality together with \eqref{deriv} in order to obtain the last estimate. This relation yields that each of the terms $U$ contributing to $Z$ can be bounded as follows:
\begin{eqnarray*}
&&\Big| U(i_1,i_2,i_3  ; a_1,a_2,a_3 ) \Big|\\
&&\quad\quad \leq (2q)! \frac{R_n(2K)}{\eps^{2K} } \eps^{2q} = (2q)! \frac{R_n(2K)}{\eps^{2K} } ( \sqrt{\eps} )^{i_1+i_2+ i_3}( \sqrt{\eps} )^{a_1+a_2+ a_3}.
\end{eqnarray*}
This yields that
\begin{eqnarray*}
&& Z\leq \frac{R_n(2K)}{\eps^{2K} } \sum_{q\geq 2K} (2q)! \sum_{i_1+i_2+i_3=2q}\sum_{a_1+a_2+a_3=2q} \left |\frac{\beta_{i_1} \alpha_{i_2,i_3}}{i_1! i_2! i_3! }\right | \times \notag \\
&&\quad\quad \quad\quad\quad \quad \left |\frac{\beta_{a_1}\alpha_{a_2,a_3}}{a_1! a_2! a_3!}\right |\times  ( \sqrt{\eps} )^{i_1+i_2+ i_3}( \sqrt{\eps} )^{a_1+a_2+ a_3} :=\frac{R_n(2K)}{\eps^{2K} }\times S  .\notag
\end{eqnarray*}
The fact that $S<\infty$ now follows from standard estimates, such as the ones appearing in \cite[end of \S 6.2.2]{DNPR}. This concludes the proof of (\ref{e:vartailk}). To prove (\ref{var tail}), it suffices to recall (\ref{r equal}) for $K=3$, and use an estimate by Bombieri-Bourgain (see \cite[Theorem 1]{BB}), stating that $|S_6(n)| = O(\mathcal N_n^{7/2}  ),$ as $\mathcal N_n\to +\infty$. \qed
\end{proof} 
We are now ready to prove Proposition \ref{prop1}.

\begin{proof}[Proposition \ref{prop1}]
\smartqed By the triangle inequality, for the l.h.s. of (\ref{bound1}) we write 
\begin{eqnarray}\label{app1}
\mathbb E\left[\left | h  (\widetilde{\mathcal L}_{n} ) - h (\widetilde{\mathcal L}_{n}[4]) \right | \right ] 
\le \mathbb E\left[\left | h  (\widetilde{\mathcal L}_{n} ) - h({\mathcal L}_{n}[4]/\sqrt{\text{Var}(\mathcal L_n)}) \right | \right ] \nonumber \\ + \mathbb E\left[\left | h({\mathcal L}_{n}[4]/\sqrt{\text{Var}(\mathcal L_n)}) - h (\widetilde{\mathcal L}_{n}[4]) \right | \right ].
\end{eqnarray}
For the first term on the r.h.s. of (\ref{app1}), since $h$ is Lipschitz, from (\ref{chaos_exp}) and Cauchy-Schwartz 
\begin{eqnarray*}
\mathbb E\left[\left | h  (\widetilde{\mathcal L}_{n} ) - h({\mathcal L}_{n}[4]/\sqrt{\text{Var}(\mathcal L_n)}) \right | \right ]\le \frac{1}{\sqrt{\text{Var}(\mathcal L_n)}} \mathbb E\left[\left| \sum_{q\ge 3} \mathcal L_n[2q] \right|  \right]\nonumber \\
\le \sqrt{\frac{\sum_{q\ge 3} \text{Var}(\mathcal L_n[2q])}{\text{Var}(\mathcal L_n)}}\ll \mathcal N_n^{-1/4},
\end{eqnarray*}
where the last upper bound follows from (\ref{igor_var}) and Lemma \ref{lem1}. For the second term on the r.h.s. of (\ref{app1}), we have again by the Lipschitz property and some standard steps
\begin{eqnarray*}
\mathbb E \left[\left | h({\mathcal L}_{n}[4]/\sqrt{\text{Var}(\mathcal L_n)}) - h (\widetilde{\mathcal L}_{n}[4]) \right | \right ]\nonumber \\ \le \left | \frac{1}{\sqrt{\text{Var}(\mathcal L_n)}} -  \frac{1}{\sqrt{\text{Var}(\mathcal L_n[4])}}\right | \mathbb E \left[ \left| \mathcal{L}_{n}[4] \right | \right]\nonumber \\
= \frac{1}{\sqrt{\text{Var}(\mathcal L_n[4])}}\left | \frac{1}{\sqrt{1 + \frac{\sum_{q\ge 3} \text{Var}(\mathcal L_n[2q])}{\text{Var}(\mathcal L_n[4])}}} - 1 \right | \mathbb E \left[ \left| \mathcal{L}_{n}[4] \right | \right]\nonumber \\
\le \left | \frac{1}{\sqrt{1 + \frac{\sum_{q\ge 3} \text{Var}(\mathcal L_n[2q])}{\text{Var}(\mathcal L_n[4])}}} - 1 \right | \ll \mathcal N_n^{-1/4},
\end{eqnarray*}
where the last bound comes from (\ref{exact_var4}) and Lemma \ref{lem1}. 
\qed
\end{proof}

\section{Proofs of Proposition \ref{prop2} and Theorem \ref{mainth1}}\label{sec:proofProp2}

Recall (\ref{LC}), then we can rewrite (\ref{eq1bella}) as
\begin{equation}\label{eq2bella}
\mathcal L_n[4] = \sqrt{\frac{E_n}{\mathcal N_n^2}}\frac{1}{\sqrt{512}}\left(p(\widehat W) +\psi_n\right),
\end{equation}
where
\begin{equation}\label{psi}
\psi_n := \frac{1}{2} \frac{1}{\mathcal N_n}\sum_{\lambda\in \Lambda_n} (|a_\lambda|^4 -2),
\end{equation}
\begin{equation}\label{W cappuccio}
\widehat W:= (W_1, W_2, W_4),
\end{equation}
and $p$ is the polynomial 
\begin{equation}\label{pol1}
p(x,y,z):=1-x^2 -4y^2 +4xy - 4z^2.
\end{equation}
The following statement is a key step in order to prove Proposition \ref{prop2}.
\begin{lemma}\label{lemGaunt}
Let $h:\mathbb R\to \mathbb R$ be a 1-Lipschitz function, define $\widehat W$ as in (\ref{W cappuccio}) for a fixed $n\in S$, and select $\eta \in [-1,1]$. Then, on $S$, 
\begin{equation}\label{eqlem}
 \left | \mathbb E\left[ h \left (p(\widehat W)\right )\right]- \mathbb E\left[ h \left (p(\widehat Z)\right )  \right] \right |\ll  \,  |\widehat{\mu_{n}}(4)-\eta|^{1/2} \vee \mathcal N_{n}^{-1/4},
\end{equation}
where the constant involving in the previous estimation is independent of $\eta$ and $h$, $p$ is the second degree polynomial defined in (\ref{pol1}) and $\widehat Z = \widehat Z(\eta) := (Z_1, Z_2, Z_4)$ is defined according to (\ref{e:punz}). 
\end{lemma}
\begin{proof}
\smartqed  We will apply an approximation argument from Ch. D\"obler's dissertation \cite{CD}. Indeed, according to \cite[Proposition 2.7.5, Corollary 2.7.6 and Lemma 2.7.7]{CD}, to every Lipschitz mapping $h$ as in the statement one can associate a collection of real-valued functions $\{h_\rho : \rho\geq 1\}$, such that the following properties are verified for every $\rho$: (i) $h_\rho$ equals the convolution of $h$ with a centered Gaussian density with variance $1/\rho^2$, (ii) $h_\rho$ is continuously infinitely differentiable, and $\| h^{(m)}_\rho \|_\infty \leq \rho^{m-1}$ (with $h_\rho^{(m)}$ denoting the $m$th derivative of $h_\rho$), and (iii) for every integrable random variable $X$, one has that $ | \mathbb{E}[ h(X)- h_\rho(X)]| \leq \rho^{-1}$. From Point (iii) it follows in particular that 
$$
\left | \mathbb E\left[ h \left (p(\widehat W)\right )\right]- \mathbb E\left[ h \left (p(\widehat Z)\right )  \right] \right |\leq \frac{2}{\rho} + \left | \mathbb E\left[ F_\rho(\widehat W) \right]- \mathbb E\left[ F_\rho (\widehat Z) \right] \right | =: \frac{2}{\rho}+B(\rho),
$$
with $F_\rho := h_\rho \circ p$. Note that $F_\rho$ is an infinitely differentiable mapping, whose partial derivatives have at most polynomial growth. This implies that we can directly apply the same interpolation and integration by parts argument one can find in \cite[Proof of Theorem 6.1.2]{NPbook}, to deduce that 
\begin{eqnarray*}
&& B(\rho) 
\le \underbrace{\sum_{i,j=1}^3 | \widehat \Sigma(i,j)-\widehat \Sigma_n(i,j)|\mathbb E[|\partial_{i,j}^2 F_\rho (\widehat W(n))|]}_{:=I_1}\nonumber \\
&& +\underbrace{\sum_{i,j=1}^3\sqrt{\mathbb E[|\partial_{i,j}^2 F_\rho(\widehat W(n))|^2]\,\mathbb E[|\widehat \Sigma_n(i,j)-\langle D\widehat W_j(n)),-DL^{-1}\widehat W_i(n)\rangle|^2]}}_{:=I_2},
\end{eqnarray*}
where $\partial^2_{i,j} := \partial^2/ \partial x_i \partial x_j$, $D$ denotes the Malliavin derivative (see \cite[Definition 1.1.8]{NPbook}), $L^{-1}$ the inverse of the infinitesimal generator of the Ornstein-Uhlenbeck semigroup (see \cite[\S 1.3]{NPbook}) and $\langle \cdot, \cdot \rangle$ stands for the inner product of an appropriate real separable Hilbert space $\mathcal{H}$ (whose exact definition is immaterial for the present proof). Standard arguments based on hypercontractivity and Point (ii) discussed above (together with the fact that $\rho\geq 1$) yield that $ E[|\partial_{i,j}^2 F_\rho(\widehat W(n))|^2]^{1/2}\leq C \rho$, for some absolute constant $C$. In view of these facts, relations (\ref{e:sign}) and (\ref{e:sig}) imply therefore that
\begin{equation}\label{Bound1} 
I_1 \ll |\widehat \mu_n(4) -\eta|.
\end{equation}
To deal with $I_2$, we can use the upper bound in \cite[formula (6.2.6)]{NPbook}, together with the fact that each $\widehat W_i(n)$ belongs to the second Wiener chaos; it hence remains to compute the fourth cumulant $k_4(\widehat W_i(n)) = \mathbb{E}[\widehat W_i(n)^4] - 3\mathbb{E}[\widehat W_i(n)^2]^2$ for every $i$ (note that these cumulants are necessarily positive). Standard computations yield that, 
\begin{eqnarray*}
k_4(W_1(n)) \ll \frac{1}{\mathcal N_n},\quad 
k_4(W_2(n)) \ll \frac{1}{\mathcal N_n}\frac{1}{\mathcal N_n}\sum_{\lambda} \frac{\lambda_1^8}{n^4},\nonumber \\
k_4(W_4(n)) \ll \frac{1}{\mathcal N_n}\frac{1}{\mathcal N_n}\sum_{\lambda} \frac{\lambda_1^4\lambda_2^4}{n^4},
\end{eqnarray*}
from which we deduce 
\begin{equation}\label{Bound2}
I_2 \ll \sqrt{ \frac{1}{\mathcal N_n}}.
\end{equation}
We have therefore proved the existence of an absolute constant $C$ such that
$$
\left | \mathbb E\left[ h \left (p(\widehat W)\right )\right]- \mathbb E\left[ h \left (p(\widehat Z)\right )  \right] \right |\leq C\left\{ \frac1\rho + \rho \gamma_n \right\},
$$
with $\gamma_n :=(2 \mathcal{N}_n^{1/2})^{-1} |\widehat \mu_n(4) -\eta|\leq 1$. Since the right-hand side of the previous inequality is maximised at the point $\rho = \gamma_n^{-1/2}$, we immediately obtain the desired conclusion.
\qed
\end{proof}
Let us now prove Proposition \ref{prop2}.
\begin{proof}[Proposition \ref{prop2}]
\smartqed
We can rewrite the l.h.s. of (\ref{ineq2}) as 
$$
  \left | \mathbb E\left[ h  \left ( \frac{p(\widehat W) + \psi_{n_j}}{\sqrt{1 + \widehat{\mu_{n_j}}(4)^2 + 34/\mathcal N_{n_j}}}\right ) - h \left (\frac{p(Z)}{\sqrt{1+\eta^2}} \right )  \right] \right |,
$$
where for $n\in S$, $\psi_n$ is given in (\ref{psi}). 
By the triangle inequality, 
\begin{eqnarray}\label{ineq2a}
\mathbb E\left[\left | h  \left ( \frac{p(\widehat W) + \psi_{n_j}}{\sqrt{1 + \widehat{\mu_{n_j}}(4)^2 + 34/\mathcal N_{n_j}}}\right ) - h \left (\frac{p(Z)}{\sqrt{1+\eta^2}} \right ) \right | \right]\nonumber \\
\le \mathbb E\left[\left | h  \left ( \frac{p(\widehat W) + \psi_{n_j}}{\sqrt{1 + \widehat{\mu_{n_j}}(4)^2 + 34/\mathcal N_{n_j}}}\right ) - h \left (\frac{p(\widehat W)}{\sqrt{1 + \widehat{\mu_{n_j}}(4)^2 + 34/\mathcal N_{n_j}}} \right ) \right | \right]\nonumber \\
+ \mathbb E\left[\left | h  \left ( \frac{p(\widehat W)}{\sqrt{1 + \widehat{\mu_{n_j}}(4)^2 + 34/\mathcal N_{n_j}}}\right ) - h \left (\frac{p(\widehat W)}{\sqrt{1 + \eta^2}} \right ) \right | \right]\nonumber \\
+ \left | \mathbb E\left[ h \left (\frac{p(\widehat W)}{\sqrt{1 + \eta^2}} \right )- h \left (\frac{p(Z)}{\sqrt{1+\eta^2}} \right )  \right] \right |\nonumber \\
=: I_{n_j} + J_{n_j} + K_{n_j}. 
\end{eqnarray}
For the first term we simply have, since $h$ is Lipschitz, 
\begin{equation}\label{i}
I_{n_j} \ll \text{Var}(\psi_{n_j}) = \frac{10}{\mathcal N_{n_j}},
\end{equation}
where the last equality is (\ref{var_psi}). Let us now deal with $J_{n_j}$. By the Lipschitz property, 
\begin{eqnarray}\label{j}
J_{n_j} \le \sqrt{1 + \widehat{\mu_{n_j}}(4)^2}\left |  \frac{1}{\sqrt{1 + \widehat{\mu_{n_j}}(4)^2 + 34/\mathcal N_{n_j}}} - \frac{1}{\sqrt{1 + \eta^2}}\right |\nonumber \\
= \sqrt{\frac{1 + \widehat{\mu_{n_j}}(4)^2}{1 + \eta^2}}\left | \frac{1}{\sqrt{1 + \frac{\widehat{\mu_{n_j}}(4)^2-\eta^2 + 34/\mathcal N_{n_j}}{1+\eta^2}}}  -1 \right |\nonumber\\
\ll |\widehat{\mu_{n_j}}(4)^2-\eta^2| + 34\mathcal N_{n_j}^{-1}
\ll ||\widehat{\mu_{n_j}}(4)|-\eta| \vee \mathcal N_{n_j}^{-1}. 
\end{eqnarray}
Finally, note that Lemma \ref{lemGaunt} and the equality in law $\mathcal M_\eta = \mathcal M_{-\eta}$ give
$$
K_n \ll ||\widehat{\mu_{n_j}}(4)|-\eta |^{1/2} \vee \mathcal N_{n_j}^{-1/4}. 
$$
Plugging the latter bound, (\ref{i}) and (\ref{j}) into (\ref{ineq2a}) we conclude the proof of Proposition \ref{prop2}.
\qed
\end{proof}

\subsection{Proof of Theorem \ref{mainth1}}\label{sec:proofmain}

\begin{proof}[Theorem \ref{mainth1}]
\smartqed
For every $j\ge 1$,  reasoning as in (\ref{eqapp}),
\begin{eqnarray*}
\mathbb E\left [\left |h(\widetilde {\mathcal L}_{n_j}) - h(\mathcal M_\eta) \right |\right ]\le \mathbb E\left[\left | h  (\widetilde{\mathcal L}_{n_j} ) - h (\widetilde{\mathcal L}_{n_j}[4]) \right | \right ]  \nonumber \\+ \mathbb E\left[\left | h  (\widetilde{\mathcal L}_{n_j}[4] ) - h (\mathcal M_\eta ) \right | \right]\nonumber\\
\ll \mathcal N_{n_j}^{-1/4} \, \vee \, | |\widehat{\mu_{n_j}}(4)| - \eta |^{1/2}, \nonumber
\end{eqnarray*}
where the last step directly follows from Proposition \ref{prop1} and Proposition \ref{prop2}. 
\qed
\end{proof}

\begin{acknowledgement}
We thank Ch. D\"obler for useful discussions, and in particular for pointing out the relevance of \cite{CD}. The research leading to this work has been supported by the grant F1R-MTH-PUL-15STAR (STARS) at the University of Luxembourg. 
\end{acknowledgement}
\section*{Appendix}
\addcontentsline{toc}{section}{Appendix}

\begin{proof}[Lemma \ref{lemChaosLeray}]
\smartqed
From \cite[Lemma 3.4]{MPRW}, we have that the chaotic expansion of $\mathcal Z_n^\varepsilon$ is 
\begin{equation}\label{leray_eps}
\mathcal Z_n^\varepsilon = \sum_{q=0}^{+\infty} \mathcal Z_n^\varepsilon[2q] = \sum_{q=0}^{+\infty} \frac{\beta^\varepsilon_{2q}}{(2q)!} \int_{\mathbb T} H_{2q}(T_n(x))\,dx,
\end{equation}
where $H_{2q}$ denotes the $2q$-th Hermite polynomial, and 
\begin{equation}\label{beta_eps}
\beta^\varepsilon_{0} = \frac{1}{2\varepsilon}\int_{-\varepsilon}^\varepsilon \phi(t)\,dt,\qquad \beta^\varepsilon_{2q} = -\frac{1}{\varepsilon} \phi(\varepsilon) H_{2q-1}(\varepsilon),\  q\ge 1,
\end{equation}
$\phi$ still denoting the Gaussian density. 
Taking the limit for $\varepsilon$ going to $0$ in (\ref{beta_eps}) we obtain the collection of coefficients (\ref{e:beta}), related to the (formal) Hermite expansion of the Dirac mass $ \delta_0$. Note that 
\begin{equation}\label{ker_fin}
\sum_{q=1}^{+\infty} \frac{(\beta_{2q})^2}{(2q)!} \int_{\mathbb T} r_n(x)^{2q}\,dx = \frac{1}{2\pi} \int_{\mathbb T} \left( \frac{1}{\sqrt{1-r_n(x)^2}} -1  \right)\,dx < +\infty,
\end{equation}
since the collection $\lbrace (\beta_{2q})^2/(2q)!\rbrace_q$ coincides with the sequence of Taylor coefficients of the function $x\mapsto 1/(2\pi \sqrt{1 - x^2})$ around zero; thanks to Lemma 5.3 in \cite{ORW} we have the finiteness of the integral. Therefore the series 
$$
\sum_{q=0}^{+\infty} \frac{\beta_{2q}}{(2q)!} \int_{\mathbb T} H_{2q}(T_n(x))\,dx,
$$
is a well-defined random variable in $L^2(\mathbb P)$, its variance being the series on the l.h.s. of (\ref{ker_fin}). 
Moreover, from \cite[22.14.16]{AS} and (\ref{ker_fin}) 
$$
\sum_{q=1}^{+\infty} \frac{(\beta_{2q}^\varepsilon - \beta_{2q})^2}{(2q)!} \int_{\mathbb T} r_n(x)^{2q}\,dx\le 2 \sum_{q=1}^{+\infty} \frac{(\beta_{2q})^2}{(2q)!} \int_{\mathbb T} r_n(x)^{2q}\,dx <+\infty,
$$
that implies, by the dominated convergence theorem, $
\mathcal Z_n^\varepsilon \mathop{\to} \mathcal Z_n,$ $\varepsilon \to 0,$ in $L^2(\mathbb P)$. 
\qed
\end{proof}

\begin{proof}[Lemma \ref{lem:4chaos}]
\smartqed
From (\ref{e:pp}) with $q=2$
\begin{eqnarray*}
\mathcal L_n[4] = \frac{\sqrt{E_n}}{128 \sqrt 2}\Big (8 \int_{\bf T} H_4(T_n(x))\,dx - \int_{\bf T} H_4(\widetilde \partial_1 T_n(x))\,dx - \int_{\bf T} H_4(\widetilde \partial_2 T_n(x))\,dx \nonumber\\
 - 8  \int_{\bf T} H_2(T_n(x)) H_2(\widetilde \partial_1 T_n(x))\,dx - 8  \int_{\bf T} H_2(T_n(x)) H_2(\widetilde \partial_2 T_n(x))\,dx \nonumber \\
-2 \int_{\bf T} H_2(\widetilde \partial_1 T_n(x))H_2(\widetilde \partial_2 T_n(x))\,dx. 
\label{eq:4chaos}
\end{eqnarray*}
Lemmas 5.2 and 5.5 in \cite{MPRW} together with some straightforward computations allow one to write, from (\ref{eq:4chaos}), 
\begin{eqnarray*}
\mathcal L_n[4] = \sqrt{\frac{E_n}{\mathcal N_n^2}}\frac{1}{128 \sqrt 2}\Big ( 8 W_1^2 - 16 W_2^2 -16 W_3^2 -32 W_4^2 \nonumber \\
+ \frac{1}{\mathcal N_n} \sum_{\lambda\in \Lambda_n} |a_\lambda|^4 \left(-8 +12 \left(\left( \frac{\lambda_1}{\sqrt n}\right)^2 + \left( \frac{\lambda_2}{\sqrt n}\right)^2  \right)^2 \right)\Big ).
\label{eq:4chaos2}
\end{eqnarray*}
Recalling that $\lambda_1^2 + \lambda_2 ^2 = n$, we obtain (\ref{eq1bella}). Let us now note that we can write 
\begin{eqnarray}\label{4chaos_repres2}
W_1^2 - 2W_2^2 -2W_3^2 -4W_4^2\nonumber \\
= \frac{1}{\mathcal N_n/2}\sum_{\lambda, \lambda'\in \Lambda_n^+} \left ( 1 -\frac{2}{n^2}(\lambda_1 \lambda_1' + \lambda_2 \lambda_2')^2  \right )(|a_\lambda |^2 -1) (|a_{\lambda'} |^2 -1).
\end{eqnarray}
Then it is immediate to compute from (\ref{4chaos_repres2})
\begin{equation}\label{media4}
\mathbb E\left[W_1^2 - 2W_2^2 -2W_3^2 -4W_4^2  \right] = -1. 
\end{equation}
Bearing in mind Lemma 4.1 in \cite{MPRW}, still from (\ref{4chaos_repres2}) some straightforward computations lead to 
\begin{equation}\label{mom_sec4}
\mathbb E\left[(W_1^2 - 2W_2^2 -2W_3^2 -4W_4^2 )^2 \right] = 2 + \widehat{\mu_n}(4)^2 +\frac{48}{\mathcal N_n}. 
\end{equation}
From (\ref{media4}) and (\ref{mom_sec4}) hence we find 
\begin{equation*}
\text{Var}(W_1^2 - 2W_2^2 -2W_3^2 -4W_4^2)= 1 + \widehat{\mu_n}(4)^2 +\frac{48}{\mathcal N_n}. 
\end{equation*}
Recalling that $(\sqrt 2 |a_\lambda|)^2$ is distributed as a chi-square random variable with two degrees of freedom, 
\begin{equation}\label{var_psi}
\text{Var}\left( \frac{1}{2} \frac{1}{\mathcal N_n}\sum_{\lambda\in \Lambda_n} |a_\lambda|^4 \right) =\frac{10}{\mathcal N_n},
\end{equation}
and moreover
\begin{equation*}
\text{Cov}\left(W_1^2 - 2W_2^2 -2W_3^2 -4W_4^2,  \frac{1}{2} \frac{1}{\mathcal N_n}\sum_{\lambda\in \Lambda_n} |a_\lambda|^4  \right) =-\frac{12}{\mathcal N_n}. 
\end{equation*}
This concludes the proof of Lemma \ref{lem:4chaos}. 
\qed
\end{proof}

%
%
%

\end{document}